\documentclass[11pt]{article}
\usepackage{amsmath, amsthm}

\usepackage{amsmath}
\usepackage{color}
\usepackage{amssymb}
\usepackage[stable]{footmisc}

\theoremstyle{plain}
\newtheorem{theorem}{\bf Theorem}[section]

\newtheorem{lemma}[theorem]{\bf Lemma}
\newtheorem{corollary}[theorem]{\bf Corollary}

\theoremstyle{definition}

\newtheorem{definition}[theorem]{\bf Definition}
\newtheorem{remark}[theorem]{\bf Remark}

\newcommand{\N}{\mathbb N}
\newcommand{\Z}{\mathbb Z}
\newcommand{\R}{\mathbb R}

\renewcommand{\t}{\, | \,}

\numberwithin{equation}{section}

\begin{document}

\title{Dual morse index estimates and application to Hamiltonian systems with P-boundary conditions}

\author{
Shanshan Tang\thanks{Partially supported by initial Scientific Research Fund of Zhejiang Gongshang University. E-mail: ss.tang@foxmail.com}\\[1ex]
 School of Statistics and Mathematics, Zhejiang Gongshang University\\
 Hangzhou 310018, P.R. China.
}

\date{}

\maketitle

\begin{abstract}
{In this paper, we study the multiplicity of Hamiltonian systems with P-boundary conditions.}
\end{abstract}



\section{Introduction and main result}
\bigskip

We consider the solutions of nonlinear Hamiltonian systems with $P$-boundary condition
\begin{equation}\label{1}
\left\{
\begin{array}{l}
\dot{x}=JH^{\prime}(t, x),\ \forall x\in\R^{2n},\\
\!x(1) = Px(0),
\end{array}\right.
\end{equation}
where $P\in Sp(2n)$ satisfies $P^{T}P=I_{2n}$, $J=\left( \begin{array}{cc}
0 \ \ & -I_{n}\\
I_{n} & 0
\end{array} \right)$
is the standard symplectic matrix, $I_{n}$, $I_{2n}$ are the identity matrices on $\R^{n}$ and $\R^{2n}$, $n$ is the positive integer.
The Hamiltonian function $H\in C^{2}(\R\times\R^{2n}, \R)$ satisfies the following conditions:
\begin{enumerate}

\item[(H)]  $H(t+1, Px) = H(t, x)$, $\forall (t, x)\in \R\times\R^{2n}$;

\smallskip
\item[$(H_{0})$] $H^{\prime}(t, 0)\equiv 0$;

\smallskip
\item[$(H_{\infty})$] There  two continuous symmetric matrix functions $B_{j}(t), j=1,2$ satisfying $P^{T}B_{j}(t+1)P = B_{j}(t)$, $i_{P}(B_{1})=i_{P}(B_{2})$ and $\nu_{P}(B_{2})=0$ such that
    \begin{equation*}
    B_{1}(t)\leq H^{\prime\prime}(t, x)\leq B_{2}(t),\ \ \forall (t, x)\ \ \text{with}\ \ \vert x\vert\geq r\ \ \text{for some large}\ \ r>0.
    \end{equation*}

\end{enumerate}

Let $\mathfrak{L}_{s}(\R^{2n})$ denotes all symmetric real $2n\times 2n$ matrices. For $A$, $B\in \mathfrak{L}_{s}(\R^{2n})$, $A\geq B$ means that $A-B$ is a semipositive definite matrix, and $A>B$ means that $A-B$ is a positive definite matrix.

\begin{theorem}\label{the:4}
Suppose that $P \in Sp(2n)$ satisfies $P^{T}P=I_{2n}$, $H$ satisfies conditions (H), $(H_{0})$, $(H_{\infty})$. Suppose $JB_{1}(t)=B_{1}(t)J$ and $B_{0}(t)=H^{\prime\prime}(t, 0)$ satisfying one of the twisted conditions
\begin{equation}\label{6}
B_{1}(t)+lI_{2n}\leq B_{0}(t)
\end{equation}
\begin{equation}\label{7}
B_{0}(t)+lI_{2n}\leq B_{1}(t)
\end{equation}
for some constant $l\geq 2\pi$. Then the system (\ref{1}) possesses at least one non-trivial {\it $P$-solution}. Furthermore, if $\nu_{P}(B_{0})=0$, the system (\ref{1}) possesses at least two non-trivial {\it $P$-solutions}.
\end{theorem}

 A solution $(1, x)$ of the problem (\ref{1}) is called {\it $P$-solution} of the Hamiltonian systems. It is a kind of generalized periodic solution of Hamiltonian systems. The problem (\ref{1}) has relation with the the closed geodesics on Riemannian manifold (cf.\cite{HuSun2}) and symmetric periodic solution or the quasi-periodic solution problem (cf.\cite{HuSun}). In addition, C. Liu in \cite{LC5} transformed some periodic boundary problem for asymptotically linear delay differential systems and some asymptotically linear delay Hamiltonian systems to $P$-boundary problems of Hamiltonian systems as above, we also refer \cite{CAMR,FDTS, HuSun1,HuWang} and references therein for the background of $P$-boundary problems in $N$-body problems.

We briefly review the general study of the problem (\ref{1}). C. Liu (cf. \cite{LC}) used the Maslov P-index theory, the Poincar$\acute{e}$ polynomial of the Conley homotopic index of isolated invariant set and saddle point reduction method to study (\ref{1}) with asymptotically linear condition. At the same time, Y. Dong (cf.\cite{dong}) combined dual Morse index with Maslov P-index to study the existence and multiplicity of (\ref{1}) under convexity condition. Since then, many papers appeared about the study of P-boundary problems (cf.\cite{donglong,LC5,LC2,liutss2,liutss3,tang1}). Besides, Y. Dong defined another dual Morse index to study the Bolza problem in \cite{dong1}. In this paper, we develop $l$-dual Morse index about $P$-boundary problem which is similar to \cite{dong1} and then use it to discuss the existence and multiple solutions of (\ref{1}). The dual Morse index theory for periodic boundary condition was studied by Girardi and Matzeu \cite{GM} for superquadratic Hamiltonian systems and by C. Liu in \cite{LC4} for subquadratic Hamiltonian systems. This theory is an application of the Morse-Ekeland index theory (cf.\cite{ek}). The index theory for convex Hamiltonian systems was established by I. Ekeland (cf.\cite{ek}), whose works are of fundamental importance in the study of convex Hamiltonian systems.

This paper is divided into 4 sections. In Section 2, we recall Maslov $P$-index and study some properties. In Section 3, we use the method in \cite{dong1,LC3} to develop $l$-dual Morse index which is suitable for $P$-boundary problems and discuss the relationship with Maslov $P$-index. Base on the study of Section 2 and 3, in Section 4, we look for {\it $P$-solution} of (\ref{1}) under twisted conditions and prove Theorem \ref{the:4}.

\bigskip

\section{Some properties of the Maslov P-index theory}
\bigskip

\ \ \ \ Maslov P-index was first studied in \cite{dong} and \cite{LC} independently for any symplectic matrix $P$ with different treatment, it was generalized by C. Liu and the author in \cite{liutss1,liutss2}. And then C. Liu used relative index theory to develop Maslov P-index in \cite{LC2} which is consistent with the definition in \cite{liutss1,liutss2}. In fact, when the symplectic matrix $P=diag \{-I_{n-\kappa}, I_{\kappa}, -I_{n-\kappa}, I_{\kappa}\}$, $0\leq\kappa\in\Z\leq n$, the $(P,\omega)$-index theory and its iteration theory were studied in \cite{donglong} and then be successfully used to study the multiplicity of closed characteristics on partially symmetric convex compact hypersurfaces in $\R^{2n}$. Here we use the notions and results in \cite{LC2,liutss1,liutss2}.

For $P\in Sp(2n)$, $B(t)\in C(\R, \mathfrak{L}_{s}(\R^{2n}))$ and satisfies $P^{T}B(t+1)P = B(t)$. If $\gamma$ is the fundamental solution of the linear Hamiltonian systems
\begin{equation}\label{2}
\dot y(t)= JB(t)y,\ \ \ y\in \R^{2n}.
\end{equation}
Then the Maslov $P$-index pair of $\gamma$ is defined as a pair of integers
\begin{equation}\label{19}
(i_{P}(B), \nu_{P}(B))\equiv (i_{P}(\gamma), \nu_{P}(\gamma))\in \Z\times \{0,1,\cdots,2n\},
\end{equation}
where $i_{P}$ is the index part and
$$ \nu_{P}=\dim \ker(\gamma(1)-P)$$
is the nullity. We also call $(i_{P}, \nu_{P})$ the Maslov P-index of $B(t)$, just as in \cite{LC2,liutss1,liutss2}. If $x$ is a {\it $P$-solution} of (\ref{1}), then the Maslov P-index of the solution $x$ is defined to be the Maslov P-index of $B(t)=H^{\prime\prime}(x(t))$ and denoted by $(i_{P}(x), \nu_{P}(x))$.

The Hilbert space $W^{1/2,2}([0, 1], \R^{2n})$ consists of all the elements of $z\in L^{2}([0, 1], \R^{2n})$ satisfying
\begin{equation*}
z(t)=\sum_{j\in\Z}\exp(2j\pi tJ)a_{j},\ \ \sum_{j\in\Z}(1+\vert j\vert)a_{j}^{2}<\infty,\ \ \ a_{j}\in \R^{2n}.
\end{equation*}
For $P\in Sp(2n)$, we define
$$W_{P}=\{z\in W^{1/2,2}([0, 1], \R^{2n})\mid z(t+1)=Pz(t)\},$$
it is a closed subspace of $W^{1/2,2}([0, 1], \R^{2n})$ and is also a Hilbert space with norm $\Vert\cdot\Vert$ and inner product $\langle\cdot, \cdot\rangle$ as in $W^{1/2,2}([0, 1], \R^{2n})$.

Let $\mathfrak{L}_{s}(W_{P})$ and  $\mathfrak{L}_{c}(W_{P})$ denote the space of the bounded selfadjoint linear operator and compact linear operator on $W_{P}$. We define two operators $A$, $B\in \mathfrak{L}_{s}(W_{P})$ by the following bilinear forms:
\begin{equation}\label{3}
\langle Ax, y \rangle = \int_{0}^{1}(-J\dot{x}(t), y(t))dt,\ \ \langle Bx, y \rangle = \int_{0}^{1}(B(t)x(t), y(t))dt.
\end{equation}
Then $B\in \mathfrak{L}_{c}(W_{P})$ (cf. \cite{LOZE}). Using the Floquet theory we have
\begin{equation}
\nu_{P}(B)=\dim\ker(A-B).
\end{equation}

Suppose that $\cdots \leq \lambda_{-j}\leq \cdots \leq \lambda_{-1}<0<\lambda_{1}\leq \cdots \leq \lambda_{j}\leq \cdots$ are all nonzero eigenvalues of the operator $A$ (count with multiplicity), correspondingly, $g_{j}$ is the eigenvector of $\lambda_{j}$ satisfying $\langle g_{j}, g_{i}\rangle=\delta_{ji}$.
We denote the kernel of the operator $A$ by $W_{P}^{0}$, specially it is exactly the space $\ker_{\R}(P-I)$.
For $m \in \N$, we define a finite dimensional subspace of $W_{P}$ by
\begin{equation*}
W^{m}_{P} = W_{m}^{-}\oplus W_{P}^{0}\oplus W_{m}^{+}
\end{equation*}
with $W_{m}^{-}=\{z\in W_{P}\vert z(t)=\sum_{j=1}^{m}a_{-j}g_{-j}(t), a_{-j}\in\R\}$ and $W_{m}^{+}=\{z\in W_{P}\vert z(t)=\sum_{j=1}^{m}a_{j}g_{j}(t),\\
a_{j}\in\R\}$.

We suppose $P_{m}$ be the orthogonal projections $P_{m}: W_{P} \to W_{P}^{m}$ for $m\in\N\cup \{0\}$. Then $\{P_{m} \mid m=0, 1, 2, \cdots\}$ be the Galerkin approximation sequence respect to $A$.

For $S\in \mathfrak{L}_{s}(W_{P})$, we denote by $M^{*}(S)$ the eigenspaces of $S$
with eigenvalues belonging to $(0, +\infty)$, $\{0\}$ and $(-\infty, 0)$ with  $* = +,0$ and $* = -$, respectively. Similarly, for any $d > 0$, we denote by $M_{d}^{*}(S)$ the $d$-eigenspaces of $S$ with eigenvalues belonging to $[d, +\infty)$, $(-d, d)$ and
$(-\infty, -d]$ with  $* = +,0$ and $* = -$, respectively. We denote $m^{*}(S)=\dim M^{*}(S)$, $m_{d}^{*}(S)=\dim M_{d}^{*}(S)$ and $S^{\sharp} = (S\vert_{Im S})^{-1}$.

\begin{theorem}\label{the:1}
Suppose $B(t)\in C(\R, \mathfrak{L}_{s}(\R^{2n}))$ and satisfies $P^{T}B(t+1)P = B(t)$ with the Maslov P-index $(i_{P}(B), \nu_{P}(B))$, for any constant $0<d<\frac{1}{4}\Vert (A - B)^{\sharp} \Vert^{-1}$, there exists an $m_{0}>0$ such that for $m\geq m_{0}$, there holds
\begin{equation}\label{35}
\begin{split}
m_{d}^{+}(P_{m}(A - B)P_{m}) &= m+\dim\ker_{\R}(P-I_{2n})-i_{P}(B)-\nu_{P}(B),\\
m_{d}^{-}(P_{m}(A - B)P_{m}) &= m+i_{P}(B),\\
m_{d}^{0}(P_{m}(A - B)P_{m}) &= \nu_{P}(B),
\end{split}
\end{equation}
where $B$ be the operator defined by (\ref{3}) corresponding to $B(t)$.
\end{theorem}

As a direct consequence, we have the following monotonicity result.
\begin{corollary}\label{coro:1}
Suppose $B_{j}(t)\in C(\R, \mathfrak{L}_{s}(\R^{2n}))$, $j=1,2$ with $P^{T}B_{j}(t+\tau)P = B_{j}(t)$ satisfy
\begin{equation}
B_{1}(t)<B_{2}(t), \ i.e.,\  B_{2}(t)-B_{1}(t) \ \ \text{is positive definite for all}\ \ t\in [0, \tau].
\end{equation}
There there holds
\begin{equation}
i_{P}(B_{1})+\nu_{P}(B_{1})\leq i_{P}(B_{2}).
\end{equation}
\end{corollary}
\begin{proof}
Let $\Gamma=\{P_{m}\}$ be the approximation scheme with respect to the operator $A$. Then by (\ref{35}), there exists $m_{0}>0$ such that if $m\geq m_{0}$, there holds
\begin{equation*}
\begin{split}
m_{d}^{-}(P_{m}(A - B_{1})P_{m}) &= m+i_{P}(B_{1}),\\
m_{d}^{-}(P_{m}(A - B_{2})P_{m}) &= m+i_{P}(B_{2}),\\
m_{d}^{0}(P_{m}(A - B_{1})P_{m}) &= \nu_{P}(B_{1}),
\end{split}
\end{equation*}
when $0<d<\frac{1}{2}\Vert B_{2}-B_{1}\Vert$. Since $A-B_{2}=(A-B_{1})-(B_{2}-B_{1})$ and $B_{2}-B_{1}$ is positive definite in $W_{P}^{m}$ and $\langle (B_{2}-B_{1})x, x\rangle\geq 2d\Vert x\Vert$. Hence we have $\langle P_{m}^{-}(A-B_{1})P_{m}x, x\rangle\leq -d\Vert x\Vert$ with
$$x\in M_{d}^{-}(P_{m}^{-}(A-B_{1})P_{m})\oplus M_{d}^{0}(P_{m}^{-}(A-B_{1})P_{m}).$$
It implies that $m+i_{P}(B_{1})+\nu_{P}(B_{1})\leq m+i_{P}(B_{2})$.

\end{proof}

\begin{remark}\label{rem:1}
From the proof of Corollary \ref{coro:1}, it is easy to show that if $B_{1}(t)\leq B_{2}(t)$ for all $t\in [0, \tau]$,
\begin{equation}
i_{P}(B_{1})\leq i_{P}(B_{2}),\ \ i_{P}(B_{1})+\nu_{P}(B_{1})\leq i_{P}(B_{2})+\nu_{P}(B_{2}).
\end{equation}
\end{remark}

\begin{definition}\label{def:1}
For $P\in Sp(2n)$, suppose $B_{j}(t) \in C(\R, \mathfrak{L}_{s}(\R^{2n}))$, $j=1,2$ satisfies $B_{j}(t+1)=(P^{-1})^{T}B_{j}(t)P^{-1}$ and $B_{1}(t)<B_{2}(t)$ for all $t\in [0, 1]$, we define
\begin{equation*}
I_{P}(B_{1}, B_{2})=\sum_{s\in [0, 1)}\nu_{P}((1-s)B_{1}+sB_{2}).
\end{equation*}
\end{definition}

\begin{theorem}\label{the:2}
For $P\in Sp(2n)$, suppose $B_{j}(t) \in C(\R, \mathfrak{L}_{s}(\R^{2n}))$, $j=1,2$ satisfies $B_{j}(t+1)=(P^{-1})^{T}B_{j}(t)P^{-1}$ and $B_{1}(t)<B_{2}(t)$ for all $t\in [0, 1]$, there holds
\begin{equation*}
I_{P}(B_{1}, B_{2})=i_{P}(B_{2})-i_{P}(B_{1}).
\end{equation*}
Hence we call $I_{P}(B_{0}, B_{1})$ the relative $P$-index of the pair $(B_{1}, B_{2})$.
\end{theorem}

In \cite{LC2}, C. Liu proved that $m_{d}^{0}(P_{m}(A - B)P_{m})$ eventually becomes a constant independent of $m$ and for large $m$, there holds
\begin{equation}
m_{d}^{0}(P_{m}(A - B)P_{m})=m^{0}(A - B).
\end{equation}
Hence
\begin{equation}\label{3}
m_{d}^{-}(P_{m}(A - B)P_{m})=m^{-}(P_{m}(A - B)P_{m}).
\end{equation}
And further,  the difference of the $d$-Morse indices $m_{d}^{-}(P_{m}(A - B)P_{m}) - m_{d}^{-}(P_{m}AP_{m})$ eventually becomes a constant independent of $m$ for large $m$, where $d>0$ is determined by the operators $A$ and $A-B$.  Then he defined the relative index by
\begin{equation}
I(A, A-B)=m_{d}^{-}(P_{m}(A - B)P_{m}) - m_{d}^{-}(P_{m}AP_{m}),\ \ m\geq m^{\ast},
\end{equation}
and got the following important results.
\begin{theorem}\label{the:6}
Suppose $B(t) \in C(\R, \mathfrak{L}_{s}(\R^{2n}))$ satisfies $B(t+1)=(P^{-1})^{T}B(t)P^{-1}$, there holds
\begin{equation}
I(A, A-B)=i_{P}(B).
\end{equation}
\end{theorem}

\begin{lemma}\label{lem:2}
Suppose $B_{j}(t)\in C(\R, \mathfrak{L}_{s}(\R^{2n}))$, $j=1,2$ satisfy $P^{T}B_{j}(t+1)P = B_{j}(t)$ and $B_{1}(t)<B_{2}(t)$ for all $t\in \R$, there holds
\begin{equation}
i_{P}(B_{2})-i_{P}(B_{1})=\sum_{s\in [0, 1)}\nu_{P}((1-s)B_{1}+sB_{2}).
\end{equation}
\end{lemma}

\begin{theorem}\label{the:5}
The Maslov $P$-index defined by (\ref{19}) as in \cite{liutss1}, the relative $P$-index defined by Definition \ref{def:1} have the following properities:
\begin{enumerate}

\item[(1)] For $P\in Sp(2n)$, $B_{j}(t)\in C(\R, \mathfrak{L}_{s}(\R^{2n}))$, $j=1,2,3$ satisfy $P^{T}B_{j}(t+1)P = B_{j}(t)$ and $B_{1}(t)<B_{2}(t)<B_{3}(t)$ for all $t\in \R$, we have
    \begin{equation*}
    I_{P}(B_{1}, B_{2})+I_{P}(B_{2}, B_{3})=I_{P}(B_{1}, B_{3}).
    \end{equation*}

\smallskip
\item[(2)] For $P\in Sp(2n)$ with $P^{T}P=I_{2n}$, $B(t)\in C(\R, \mathfrak{L}_{s}(\R^{2n}))$ satisfies $P^{T}B(t+1)P = B(t)$, there exist $s_{0}>0$ such that for any $s\in (0, s_{0}]$, we have
    \begin{equation*}
    \begin{split}
    \nu_{P}(B+sI_{2n})&=0=\nu_{P}(B-sI_{2n}),\\
    i_{P}(B-sI_{2n})&=i_{P}(B),\\
    i_{P}(B+sI_{2n})&=i_{P}(B)+\nu_{P}(B).
    \end{split}
    \end{equation*}
In particular, if $\nu_{P}(B)=0$, we have $ i_{P}(B+sI_{2n})=i_{P}(B)$ for $s\in (0, s_{0}]$.
\end{enumerate}

\end{theorem}

\begin{proof}
(1) follows from Theorem \ref{the:2} immediately.

From Theorem \ref{the:2}, we have $i_{P}(B+I_{2n})=I_{P}(B, B+I_{2n})-i_{P}(B)$. By Lemma \ref{lem:2}, we see that $I_{P}(B, B+I_{2n})=\sum_{s\in [0, 1)}\nu_{P}(B+sI_{2n})$ is finite. So there is some $s_{0}$ such that $\nu_{P}(B+sI_{2n})=0$ for $s\in (0, s_{0}]$, and
\begin{equation}\label{20}
i_{P}(B+sI_{2n})=i_{P}(B)+\sum_{\lambda\in [0, 1)}\nu_{P}(B+\lambda sI_{2n})=i_{P}(B)+\nu_{P}(B).
\end{equation}

Similarly, $i_{P}(B-I_{2n}, B)=i_{P}(B)-i_{P}(B-I_{2n})=\sum_{s\in [0, 1)}\nu_{P}(B-(1-s)I_{2n})$ is finite, so there is some $s_{0}$ such that $\nu_{P}(B-sI_{2n})=0$ for $s\in (0, s_{0}]$, and
\begin{equation}
i_{P}(B-sI_{2n})=i_{P}(B)-\sum_{\lambda\in [0, 1)}\nu_{P}(B-(1-\lambda)sI_{2n})=i_{P}(B).
\end{equation}

If $\nu_{P}(B)=0$, by (\ref{20}) we have $ i_{P}(B+sI_{2n})=i_{P}(B)$ for $s\in (0, s_{0}]$.

\end{proof}

\bigskip

\section{Dual morse index theory for linear Hamiltonian systems with $P$-boundary conditions}
\bigskip

Recall that the Hilbert space $W_{P}=\{z\in W^{1/2,2}([0, 1], \R^{2n})\mid z(t+1)=Pz(t)\}$ with norm $\Vert\cdot\Vert$ and inner product $\langle\cdot, \cdot\rangle$. Let $L_{P}=\{z\in L^{2}([0, 1], \R^{2n})\mid z(t+1)=Pz(t)\}$ with with norm $\Vert\cdot\Vert_{2}$ and inner product $\langle\cdot, \cdot\rangle_{2}$. By the well-known Sobolev embedding theorem, the embedding $j: W_{P}\to L_{P}$ is compact.
For $P\in Sp(2n)$ with $P^{T}P=I_{2n}$, we define an operator $A: L_{P}\to L_{P}$ with domain $W_{P}$ by $A=-Jd/dt$. The spectrum of $A$ is isolated. Let $l\notin \sigma(A)$ be so large such that $B(t)+lI_{2n}>0$. Then the operator $\Lambda_{l}=A+lI_{2n}: W_{P}\to L_{P}$ is invertible and its inverse is compact. We define a quadratic form in $L_{P}$ by
\begin{equation}
Q_{l, B}^{\ast}(u, v)=\int_{0}^{1}(C_{l}(t)u(t), v(t))-(\Lambda_{l}^{-1}u(t), v(t))\ \ \forall u, v\in L_{P},
\end{equation}
where $C_{l}(t)=(B(t)+lI_{2n})^{-1}$. Define $Q^{\ast}_{l, B}(u)=Q^{\ast}_{l, B}(u, u)$. We define the operator $C_{l}: L_{P}\to L_{P}$ by
\begin{equation*}
\langle C_{l}u, v\rangle_{2}=\int_{0}^{1}(C_{l}(t)u(t), v(t))dt.
\end{equation*}
Since $C_{l}(t)$ is positive definite, $C_{l}$ is an isomorphism and $\langle C_{l}u, u\rangle_{2}$ defines a Hilbert space structure on $L$ which is equivalent to the standard one. Endowing $L_{P}$ with the inner product $\langle C_{l}u, u\rangle_{2}$, $\Lambda_{l}^{-1}$ is a self-adjoint and compact operator and applying to $\Lambda_{l}^{-1}$ the spectral theory of compact self-adjoint operators on Hilbert space, we see there is a basis $e_{j}$, $j\in\N$ of $L_{P}$, and an eigenvalue sequence $\mu_{j}\to 0$ in $\R$ such that
\begin{equation}\label{32}
\begin{split}
\langle C_{l}e_{i}, e_{j}\rangle_{2}&=\delta_{ij},\\
\langle \Lambda_{l}^{-1}e_{j}, u\rangle_{2}&=\langle C_{l}\mu_{j}e_{j}, u\rangle_{2},\ \ \forall u\in L_{P}.
\end{split}
\end{equation}
Hence, expressing any vector $u\in L$ as $u=\sum_{j=1}^{\infty}\xi_{j}e_{j}$,
\begin{align*}
\begin{split}
Q_{l, B}^{\ast}(u)&=\int_{0}^{1}(C_{l}(t)u(t), u(t))dt-\int_{0}^{1}(\Lambda_{l}^{-1}u(t), v(t))dt\\
&=\sum_{j=1}^{\infty}\xi_{j}^{2}-\sum_{j=1}^{\infty}\mu_{j}\xi_{j}^{2}=\sum_{j=1}^{\infty}(1-\mu_{j})\xi_{j}^{2}.
\end{split}
\end{align*}

Define
\begin{equation*}
\begin{split}
L^{+}_{l}(B)&=\{\sum_{j=1}^{\infty}\xi_{j}e_{j}\mid \xi_{j}=0 \ \ \text{if}\ \  1-\mu_{j}\leq 0\},\\
L^{0}_{l}(B)&=\{\sum_{j=1}^{\infty}\xi_{j}e_{j}\mid \xi_{j}=0 \ \ \text{if}\ \  1-\mu_{j}\neq 0\},\\
L^{-}_{l}(B)&=\{\sum_{j=1}^{\infty}\xi_{j}e_{j}\mid \xi_{j}=0 \ \ \text{if}\ \  1-\mu_{j}\geq 0\}.
\end{split}
\end{equation*}
Observe that $L^{+}_{l}(B)$, $L^{0}_{l}(B)$ and $L^{-}_{l}(B)$ are $Q_{l, B}^{\ast}$-orthogonal, and $L_{P}=L^{+}_{l}(B)\oplus L^{0}_{l}(B)\oplus L^{-}_{l}(B)$. Since $\mu_{j}\to 0$ when $j\to\infty$, all the coefficients $1-\mu_{j}$ are positive except a finite number. It implies that both $L^{+}_{l}(B)$ and $L^{0}_{l}(B)$ are finite subspaces.
\begin{definition}\label{def:2}
 For $P\in Sp(2n)$ with $P^{T}P=I_{2n}$, $B(t)\in C(\R, \mathfrak{L}_{s}(\R^{2n}))$ satisfy $P^{T}B(t+1)P = B(t)$, $l\in\R$ with $B(t)+lI_{2n}>0$, we define
\begin{equation}
i_{l}^{\ast}(B)=\dim L^{-}_{l}(B),\ \ \ \nu_{l}^{\ast}(B)=\dim L^{0}_{l}(B).
\end{equation}
\end{definition}
We call $i_{l}^{\ast}(B)$ and $\nu_{l}^{\ast}(B)$ the $l$-dual Morse index and $l$-dual nullity of $B$ respectively.

\begin{theorem}\label{the:7}
Under the conditions of Theorem \ref{the:2} and Definition \ref{def:2}, we have
\begin{equation*}
\nu_{l}^{\ast}(B)=\nu_{P}(B),\ \ I_{P}(B_{1}, B_{2})=i_{l}^{\ast}(B_{2})-i_{l}^{\ast}(B_{1}).
\end{equation*}
\end{theorem}

\begin{proof}
We follow the idea in \cite{dong1} to prove it.

By definitions, $L^{+}_{l}(B)$, $L^{0}_{l}(B)$ and $L^{-}_{l}(B)$ are $Q_{l, B}^{\ast}$-orthogonal, and satisfy $L_{P}=L^{+}_{l}(B)\oplus L^{0}_{l}(B)\oplus L^{-}_{l}(B)$. For every $u\in L^{0}_{l}(B)$, we have
\begin{equation*}
Q_{l, B}^{\ast}(u, v)=0,\ \ \forall \ v\in L_{P}.
\end{equation*}
So
\begin{equation*}
C_{l}(t)u(t)-\Lambda_{l}^{-1}u(t)=0.
\end{equation*}
Set $x=\Lambda_{l}^{-1}u$. Applying $C_{l}(t)=(B(t)+lI_{2n})^{-1}$ to both sides and using the equalities $\Lambda_{l}=-J\frac{d}{dt}+l$ and $u=\Lambda_{l}x$, we obtain
\begin{equation*}
-J\dot{x}(t)+lx(t)-(B(t)+lI_{2n})x(t)=0.
\end{equation*}
That is,
\begin{equation}\label{31}
\dot{x}(t)=JB(t)x(t).
\end{equation}
Hence $\nu_{l}^{\ast}(B)$ is the dimension of $\ker(\gamma(1)-P)$, where $\gamma(t)$ is the fundamental solution of (\ref{31}) and $\nu_{l}^{\ast}(B)=\nu_{P}(B)$.

We carry out the proof of the second equality in several steps.

\smallskip
{\bf Step 1.} We show that if $X$ is a subspace of $L_{P}$ such that $Q_{l, B}^{\ast}(u, u)<0$ for every $u\in X\setminus 0$, then $\dim X\leq i_{l}^{\ast}(B)$.

In fact, suppose $e_{1}, \dots, e_{r}$ be a basis of $X$, we have the decomposition $e_{i}=e_{i}^{-}+e_{i}^{\ast}$ with $e_{i}^{-}\in  L^{-}_{l}(B)$, $e_{i}^{\ast}\in L^{+}_{l}(B)\oplus L^{0}_{l}(B)$.

Suppose there exist numbers $\alpha_{i}\in\R$ which are not all zero, such that $\sum_{i=1}^{r}\alpha_{i}e_{i}^{-}=0$.

Set $e=\sum_{i=1}^{r}\alpha_{i}e_{i}$, then $e\in X\setminus 0$ and $Q_{l, B}^{\ast}(e, e)<0$; at the same time, $e=\sum_{i=1}^{r}\alpha_{i}e_{i}^{\ast}\in L^{+}_{l}(B)\oplus L^{0}_{l}(B)$ and $Q_{l, B}^{\ast}(e, e)\geq 0$, a contradiction.

So $\{e_{i}^{-}\}_{i=1}^{r}$ is linear independent and $i_{l}^{\ast}(B)\geq r=\dim X$.

\smallskip
{\bf Step 2.} For $B_{j}(t)\in C(\R, \mathfrak{L}_{s}(\R^{2n}))$, $j=1, 2$ satisfy $P^{T}B_{j}(t+1)P = B_{j}(t)$ and $B_{1}(t)<B_{2}(t)$. Set $i(\lambda)=i_{l}^{\ast}((1-\lambda)B_{1}+\lambda B_{2})$ for $\lambda\in [0, 1]$. Then $i(\lambda_{2})\geq i(\lambda_{1})+\nu(\lambda_{1})$.

In fact, set $A_{i}=(1-\lambda_{i})B_{1}+\lambda_{i}B_{2}$ for $i=1,2$. We only need to prove that
\begin{equation*}
Q_{l, A_{2}}^{\ast}(u, u)<0,\ \ \forall\ u\in  L^{-}_{l}(A_{1})\oplus L^{0}_{l}(A_{1})\setminus 0.
\end{equation*}
Take any $u=u^{0}+u^{-}$ with $u^{0}\in L^{0}_{l}(A_{1})$, $u^{-}\in L^{-}_{l}(A_{1})$. Note that $$A_{2}-A_{1}=(\lambda_{1}-\lambda_{2})(B_{1}-B_{2})>0$$.

If $u^{-}\neq 0$, we have
\begin{equation*}
\begin{split}
Q_{l, A_{2}}^{\ast}(u, u)\leq Q_{l, A_{1}}^{\ast}(u, u)&=Q_{l, A_{1}}^{\ast}(u^{-}, u^{-})+Q_{l, A_{1}}^{\ast}(u^{0}, u^{0})\\
&=Q_{l, A_{1}}^{\ast}(u^{-}, u^{-})<0.
\end{split}
\end{equation*}

If $u^{-}=0$, set $x^{0}=\lambda_{l}^{-1}u^{0}$, then $u^{0}=\lambda_{l}x^{0}$ and $x^{0}$ is a nontrivial solution of
\begin{equation*}
J\dot{x}(t)+A_{1}(t)x(t)=0, \ \ x(1)=Px(0).
\end{equation*}
So $x^{0}(t)\neq 0$ for every $t\in [0, 1]$, and $u^{0}=(A_{1}(t)+lI_{2n})x^{0}(t)\neq 0$ for a.e. $t\in (0, 1)$. Hence
\begin{equation*}
\begin{split}
\frac{1}{\lambda_{1}-\lambda_{2}}Q_{l, A_{2}}^{\ast}(u, u)&=\frac{1}{\lambda_{1}-\lambda_{2}}(Q_{l, A_{2}}^{\ast}(u^{0}, u^{0})-Q_{l, A_{1}}^{\ast}(u^{0}, u^{0}))\\
&=\int_{0}^{1}(\frac{(B_{2}(t)-B_{1}(t))u^{0}(t)}{(A_{2}(t)+lI_{2n})(A_{1}(t)+lI_{2n})}, u^{0}(t))dt\\
&=\int_{0}^{1}(\frac{(B_{2}(t)-B_{1}(t))x^{0}(t)}{A_{2}(t)+lI_{2n}}, (A_{1}(t)+lI_{2n})x^{0}(t))dt.
\end{split}
\end{equation*}
If $\lambda_{1}=\lambda_{2}$, we have $A_{1}(t)=A_{2}(t)$ and the last integral is
\begin{equation*}
\int_{0}^{1}((B_{2}(t)-B_{1}(t))x^{0}(t), x^{0}(t))dt>0.
\end{equation*}
Hence, if $\lambda_{2}$ is close to $\lambda_{1}$ and $\lambda_{2}>\lambda_{1}$, we have $Q_{l, A_{2}}^{\ast}(u, u)<0$. So for $\lambda_{2}>\lambda_{1}$ and $\lambda_{2}$ is close to $\lambda_{1}$, we have $i(\lambda_{2})\geq i(\lambda_{1})+\nu(\lambda_{1})$.

\smallskip
{\bf Step 3.} For any $\lambda\in [0, 1)$, we have $i(\lambda+0)=i(\lambda)+\nu(\lambda)$, where $i(\lambda+0)$ is the right limit of $i(\lambda^{\prime})$ at $\lambda$, $\nu_{\lambda}=\nu_{l}^{\ast}((1-\lambda)B_{1}+\lambda B_{2})$.

In fact, we have $i(\lambda)+\nu(\lambda)\leq i(\lambda+0)$ by Step 2. So we only need to prove that $i(\lambda)+\nu(\lambda)\geq i(\lambda+0)$. Set $d=i(\lambda+0)$. There exists $\lambda^{\prime}>\lambda$ such that $i(s)=d$ and $\nu(s)=0$ for $s\in (\lambda, \lambda^{\prime})$. Set $C(s)=((1-s)B_{1}+sB_{2}+lI)^{-1}$. Similar to (\ref{32}), we have
\begin{equation}\label{33}
\begin{split}
\langle C(s)e_{i}^{s}, e_{j}^{s}\rangle_{2}&=\delta_{ij},\\
\langle \Lambda_{l}^{-1}e_{j}^{s}, u\rangle_{2}&=\langle C(s)\mu_{j}^{s}e_{j}^{s}, u\rangle_{2},\ \ \forall u\in L_{P}.
\end{split}
\end{equation}
Since $C(s)\geq (B_{2}(t)+lI)^{-1}$ for $s\in [0, 1]$, the sequence $\{e_{j}^{s}\}$ is bounded in $L_{P}$ and $\mu_{j}^{s}=\langle \Lambda_{l}^{-1}e_{j}^{s}, e_{j}^{s}\rangle_{2}$ is bounded in $\R$ for $j=1,\dots, d$. So there exist $s_{k}\in (\lambda, \lambda^{\prime})$ such that $s_{k}\to \lambda+0$, $e^{s_{k}}_{j}\rightharpoonup e_{j}$ in $L_{P}$, $\mu_{j}^{s_{k}}\to \mu_{j}$ in $\R$ and $\Lambda_{l}^{-1}e^{s_{k}}_{j}\to \Lambda_{l}^{-1}e_{j}$.

Taking the limit in (\ref{33}) we obtain $\langle C(\lambda)e_{i}, e_{j}\rangle_{2}=\delta_{ij}$ and $\Lambda_{l}^{-1}e_{j}=C(\lambda)\mu_{j}e_{j}$ for $j=1,\dots, d$. Again for $j=1,\dots, d$, since $i(s)=d$ for $s\in (\lambda, \lambda^{\prime})$, by definition we have $1+\mu_{j}^{s_{k}}<0$ and $1/\mu_{j}^{s_{k}}$ is bounded in $\R$. Hence
\begin{equation*}
e^{s_{k}}_{j}=\frac{1}{\mu_{j}^{s_{k}}}C(s_{k})^{-1}\Lambda_{l}^{-1}e^{s_{k}}_{j}\to \frac{1}{\mu_{j}}C(\lambda)^{-1}\Lambda_{l}^{-1}e_{j}=e_{j}
\end{equation*}
in $L_{P}$. It follows that $\{e_{i}\}_{i=1}^{d}$ is linearly independent and for every $u=\sum_{j=1}^{d}\alpha_{j}e_{j}$, since $\sum_{j=1}^{d}\alpha_{j}e_{j}^{s_{k}}\to u$ in $L_{P}$ and
\begin{equation*}
Q_{l, (1-s_{k})B_{1}+s_{k}B_{2}}^{\ast}(\sum_{j=1}^{d}\alpha_{j}e_{j}^{s_{k}}, \sum_{j=1}^{d}\alpha_{j}e_{j}^{s_{k}})<0,
\end{equation*}
taking the limit as $s_{k}\to \lambda+0$, we have $Q_{l, (1-\lambda)B_{1}+\lambda B_{2}}^{\ast}(u, u)\leq 0$. In a way similar to the proof of Step 1, this implies $i(\lambda)+\nu(\lambda)\geq d=i(\lambda+0)$.

\smallskip
{\bf Step 4.} The function $i(\lambda)$ is left continuous for $\lambda\in (0, 1]$ and continuous for $\lambda\in (0, 1)$ with $\nu_{\lambda}=0$.

In fact, from Step 2 and 3 we only need to show $i(\lambda)\leq i(\lambda-0)$. Let $e_{1}, \dots, e_{k}$ be a basis of $L^{-}(\lambda):=L^{-}_{l}((1-\lambda)B_{1}+\lambda B_{2})$, and
\begin{equation*}
S_{1}:=\{(\alpha_{1}, \dots, \alpha_{d})\in \R^{d}\mid \sum_{i=1}^{k}\alpha_{i}^{2}=1\}.
\end{equation*}
Then
\begin{equation*}
\begin{split}
f(s,\alpha_{1}, \dots, \alpha_{d}):&=Q_{l, (1-s)B_{1}+sB_{2}}^{\ast}(\sum_{j=1}^{d}\alpha_{j}e_{j}, \sum_{j=1}^{d}\alpha_{j}e_{j})\\
&=\int_{0}^{1}[(((1-s)B_{1}(t)+sB_{2}(t)+lI_{2n})^{-1}\sum_{j=1}^{d}\alpha_{j}e_{j}, \sum_{j=1}^{d}\alpha_{j}e_{j})-(\Lambda_{l}^{-1}\sum_{j=1}^{d}\alpha_{j}e_{j}, \sum_{j=1}^{d}\alpha_{j}e_{j})]dt
\end{split}
\end{equation*}
is continuous in $[0, 1]\times S_{1}$. Since $f(\lambda,\alpha_{1}, \dots, \alpha_{d})<0$ for $(\alpha_{1}, \dots, \alpha_{d})\in S_{1}$, we have $f(s,\alpha_{1}, \dots, \alpha_{d})<0$ for $(\alpha_{1}, \dots, \alpha_{d})\in S_{1}$ and $s$ close enough to $\lambda$.

From Step 1, we have $i(\lambda)\leq i(s)$ for $s$ close to $\lambda$. Hence $i(\lambda)\leq i(\lambda-0)$. In conclusion,
\begin{equation*}
\begin{split}
i_{l}^{\ast}(B_{2})&=i_{l}^{\ast}(B_{1})+\sum_{0\leq \lambda<1}\nu_{l}((1-\lambda)B_{1}+\lambda B_{2})\\
&=i_{l}^{\ast}(B_{1})+\sum_{0\leq \lambda<1}\nu_{P}((1-\lambda)B_{1}+\lambda B_{2})=i_{l}^{\ast}(B_{1})+I_{P}(B_{1}, B_{2}).
\end{split}
\end{equation*}

\end{proof}

Further, if $P$ satisfies $P^{k}=I_{2n}$ for some $k\in\R$, we can obtain the specific formula of $i_{l}^{\ast}(B)$ by the method used in \cite{LC3}.
\begin{theorem}\label{the:3}
Suppose that $P \in Sp(2n)$ satisfies $P^{T}P=I_{2n}$ and $P^{k}=I_{2n}$ for some $k\in\R$, under the conditions of Theorem \ref{the:7}, there holds
\begin{equation}
\nu_{l}^{\ast}(B)=\nu_{P}(B),\ \ i_{l}^{\ast}(B)=2mn+i_{P}(B)-M,
\end{equation}
where $M$ is independent of $B$ and satisfies
\begin{equation}
2n(m-1-[\frac{[l/2\pi]}{k}])\leq M\leq 2n(m-[\frac{[l/2\pi]}{k}]).
\end{equation}
\end{theorem}

\begin{proof}
For $P\in Sp(2n)$ with $P^{k}=I$, we can regard $W_{P}$ as
\begin{equation*}
W_{P}=\{z\in W^{1/2,2}(S_{k}, \R^{2n})\mid z(t+1)=Pz(t)\},\ \ S_{k}=\R/k\Z,
\end{equation*}
it is a closed subspace of $W^{1/2,2}(S_{k}, \R^{2n})$ and is also a Hilbert space with norm $\Vert\cdot\Vert$ and inner product $\langle\cdot, \cdot\rangle$ as in $W^{1/2,2}(S_{k}, \R^{2n})$.

By a direct computation, we see that $z\in W_{P}$ iff $\in W^{1/2,2}(S_{k}, \R^{2n})$ and $a_{0}$ is an eigenvector of the eigenvalue $1$ of $P$ and $a_{j}=\alpha_{j}+J\beta_{j}$, $a_{-j}=\alpha_{j}-J\beta_{j}$ with $\alpha_{j}-\sqrt{-1}\beta_{j}$ being an eigenvector of the eigenvalue $e^{2j\pi\sqrt{-1}/k}$ of $P^{-1}$ for $j\in\Z$. We set
\begin{equation}
W_{P, s}=\{z\in W_{P}\mid z(t)=\sum_{\vert j\vert=(s-1)k+1}^{sk}\exp(\frac{2j\pi tJ}{k})a_{j}\},\ \ s\in\N
\end{equation}
and
\begin{equation}
W_{P}^{m}=\bigoplus_{s=0}^{m}W_{P, s}.
\end{equation}
Hence the dimension of $W_{P}^{0}$ is exactly $\dim\ker_{\R}(P-I_{2n})$.

We define a quadratic form in $W_{P}^{m}$ by
\begin{equation}
\begin{split}
Q_{m}(x, y)&=\int_{0}^{1}(\Lambda_{l}x(t), y(t))-(C_{l}^{-1}(t)x(t), y(t))\\
&=\int_{0}^{1}[(-J\dot{x}(t), y(t))-(B(t)x(t), y(t))]dt, \ \ \forall x, y\in W_{P}^{m},
\end{split}
\end{equation}
and define a functional $Q_{m}: W_{P}^{m}\to \R$ by $Q_{m}(x)=Q_{m}(x, x)$.
We define two linear operators $A_{l}$ and $B_{l}$ from $W_{P}^{m}$ onto its dual space $(W_{P}^{m})^{\prime}\cong W_{P}^{m}$ by
\begin{equation*}
\begin{split}
\langle A_{l}x, y\rangle&=\int_{0}^{1}(-J\dot{x}(t)+lx(t), y(t))dt,\\
\langle B_{l}x, y\rangle&=\int_{0}^{1}((B(t)+l)x(t), y(t))dt,\ \ \forall x, y\in W_{P}^{m}.
\end{split}
\end{equation*}
Since $B(t)+lI_{2n}$ is positive definite, we define $\langle\cdot, \cdot\rangle_{m}:=\langle B_{l}\cdot, \cdot\rangle$ which is a new inner product in $W_{P}^{m}$. We consider the eigenvalues $\eta_{j}\in\R$ with respect to the inner $\langle\cdot, \cdot\rangle_{m}$, that is
\begin{equation}\label{4}
A_{l}x_{j}=\eta_{j}B_{l}x_{j}
\end{equation}
for some $x_{j}\in W_{P}^{m}\setminus\{0\}$. Suppose $\eta_{1}\leq\eta_{2}\leq\cdots\leq\eta_{h}$ with $h=\dim W_{P}^{m}=2m+\dim\ker_{\R}(P-I)$ (each eigenvalue is counted with multiplicity), the corresponding eigenvectors $v_{1}$, $\dots$, $v_{h}$ which construct a new basis in $W_{P}^{m}$  satisfy
\begin{equation}\label{2}
\begin{split}
\langle v_{i}, v_{j}\rangle_{m}&=\delta_{ij},\\
\langle A_{l}v_{i}, v_{j}\rangle_{m}&=\eta_{i}\delta_{ij},\\
Q_{m}(v_{i}, v_{j})&=(\eta_{i}-1)\delta_{ij}.
\end{split}
\end{equation}
The Morse indices $m^{\ast}(Q_{m})$, $\ast=+,0,-$ denote the dimension of maximum positive subspace, kernel space and maximum negative subspace of $Q_{m}$ in $W_{P}^{m}$ respectively.
By (\ref{2}), we have
\begin{equation*}
\begin{split}
m^{+}(Q_{m})&=^{\sharp}\{\eta_{j}\mid 1\leq j\leq h, \eta_{j}>1\},\\
m^{0}(Q_{m})&=^{\sharp}\{\eta_{j}\mid 1\leq j\leq h, \eta_{j}=1\},\\
m^{-}(Q_{m})&=^{\sharp}\{\eta_{j}\mid 1\leq j\leq h, \eta_{j}<1\}.
\end{split}
\end{equation*}

By Theorem \ref{the:1} and (\ref{3}), for $m>0$ large enough we have
\begin{equation}
m^{-}(Q_{m})=2mn+i_{P}(B),\ \ m^{0}(Q_{m})=\nu_{P}(B).
\end{equation}
We define $Q_{l, m}^{\ast}=Q_{l}^{\ast}\vert_{W_{P}^{m}}$ and
\begin{equation*}
i_{l, m}^{\ast}(B)=m^{-}(Q_{l, m}^{\ast}),\ \ \nu_{l, m}^{\ast}(B)=m^{0}(Q_{l, m}^{\ast}).
\end{equation*}
By the argument in \cite{GM}, We have
\begin{equation*}
i_{l, m}^{\ast}(B)\to i_{l}^{\ast}(B), \ \ \ \nu_{l, m}^{\ast}(B)\to \nu_{l}^{\ast}(B),\ \ \ \text{as}\ m\to \infty.
\end{equation*}
Let $v_{j}^{\prime}=A_{l}v_{j}$ for $j=1,2,\dots,h$. Then it is a basis of $W_{P}^{m}$. It is a basis of $W_{P}^{m}$ and it is $Q_{l, m}^{\ast}$-orthogonal, that is
\begin{equation*}
Q_{l, m}^{\ast}(v_{i}^{\prime}, v_{j}^{\prime})=0,\ \ \text{if}\ \ i\neq j.
\end{equation*}
Hence $m^{-}(Q_{l, m}^{\ast})$ equals the number of negative $Q_{l, m}^{\ast}(v_{i}^{\prime})$. As a consequence of (\ref{4}) and (\ref{2}), it easily follows that
\begin{equation}
Q_{l, m}^{\ast}(v_{i}^{\prime})=\eta_{i}(\eta_{i}-1),
\end{equation}
which is negative if and only if $0<\eta_{i}<1$.
If one replaces the inner product $\langle\cdot, \cdot\rangle_{m}$ by the usual one, that is, one replaces the matrix $B(t)+lI$ by the identity $I$, the eigenvalues $\eta_{j}$s are replaced by the eigenvalues $\lambda_{j}$s of $A_{l}$. It is easy to check that there is a corresponding between the signs of $\{\eta_{1},\dots,\eta_{h}\}$. More precisely, one has
\begin{align}
\lambda_{1}\leq\cdots\lambda_{r}\leq 0\leq\lambda_{r+1}\ \ \text{for some}\ r\in\{1,\dots,h\}\ \Leftrightarrow \ \ \eta_{1}\leq\cdots\leq\eta_{r}\leq 0\leq\eta_{r+1}\\
\text{and}\ \ \lambda_{r}=0\ \ \Leftrightarrow\ \ \eta_{r}=0.
\end{align}
So the total multiplicity of negative $\eta_{j}$s equal the total multiplicity of negative $\lambda_{j}$s. But we have
\begin{equation}
\lambda_{\kappa}=2\kappa\pi+l,\ \ -mk\leq\kappa\leq mk,
\end{equation}
and when $\kappa=0$, the multiplicity of $\lambda_{\kappa}$ is $\dim\ker_{\R}(P-I)$. Besides, the total multiplicity of $\lambda_{\kappa}$, $\pm\kappa\in [(s-1)k+1, sk]$ is $2n$ for $1\leq s\leq m$. Suppose the total multiplicity of the negative $\lambda_{\kappa}$ is $M$, it is determined by $m$ and $l$ and independent of $B(t)$. From the above argument, we have the following estimation:
\begin{equation*}
2n(m-1-[\frac{[l/2\pi]}{k}])\leq M\leq 2n(m-[\frac{[l/2\pi]}{k}]).
\end{equation*}
Hence the total multiplicity of $\lambda_{\kappa}\in (0, 1)$ is $m^{-}(Q_{m})-M$, and by definition,
\begin{equation}
i_{l, m}^{\ast}(B)=m^{-}(Q_{m})-M=2mn+i_{P}(B)-M
\end{equation}
for $m>0$ large enough.
\end{proof}

\begin{corollary}\label{coro:2} Under the condition of Theorem \ref{the:2}, there holds
\begin{equation}\label{5}
I_{P}(B_{0}, B_{1})=i_{l}^{\ast}(B_{1})-i_{l}^{\ast}(B_{0})\ \ \text{for}\ \ l>0 \ \ \text{such that}\ \ B_{j}(t)+lI>0.
\end{equation}
\end{corollary}
\begin{proof}
From Theorem \ref{the:2} and Theorem \ref{the:3} we get (\ref{5}).
\end{proof}

\bigskip

\section{Proof of Theorem \ref{the:4}}
\bigskip

In order to prove Theorem \ref{the:4} we need a lemma. Let $E$ be a Banach space and $f\in C^{2}(E, \R)$. Set $K=\{x\in E\mid f^{\prime}(x)=0\}$ and $f_{a}=\{x\in L_{P}\mid f(x)\leq a\}$. If  $f^{\prime}(p)=0$ and $c=f(p)$, we say that $p$ is a critical point of $f$ and $c$ is critical value. Otherwise, we say that $c\in\R$ is a regular value of $f$. For any $p\in E$, $f^{\prime\prime}(p)$ is a self-adjoint operator, the Morse index of $p$ is defined as the dimension of the negative space corresponding to the spectral decomposing, and is denoted by $m^{-}(f^{\prime\prime}(p))$. We also set $m^{0}(f^{\prime\prime}(p))=\dim\ker f^{\prime\prime}(p)$.

\begin{lemma}\label{lem:1}
Let $f\in C^{2}(E, \R)$ satisfy the (P.S) condition $f^{\prime}(0)=0$ and there exists
\begin{equation*}
r\notin [m^{-}(f^{\prime\prime}(0)), m^{-}(f^{\prime\prime}(0))+m^{0}(f^{\prime\prime}(0))]
\end{equation*}
with $H_{q}(E, f_{a}; \R)\cong \delta_{q, r}\R$. Then $f$ at least one nontrivial critical point $u_{1}\neq 0$. Moreover, if $m^{0}(f^{\prime\prime}(0))=0$ and $m^{0}(f^{\prime\prime}(u_{1}))\leq\vert r- m^{-}(f^{\prime\prime}(0))\vert$, then $f$ has one more nontrivial critical point $u_{2}\neq u_{1}$.
\end{lemma}

\begin{proof}
Without loss of generality, we can suppose that $H(t, 0)=0$. By the condition $(H_{\infty})$ and Remark \ref{rem:1}, we find that $i_{P}(B_{1})+\nu_{P}(B_{1})\leq i_{P}(B_{2})+\nu_{P}(B_{2})$, so we have $\nu_{P}(B_{1})=0$. Firstly, we prove that under the conditions (\ref{6}) or (\ref{7}), it holds that
\begin{equation*}
i_{P}(B_{1})\notin [i_{P}(B_{0}), i_{P}(B_{0})+\nu_{P}(B_{0})].
\end{equation*}

More precisely, under the condition (\ref{6}), there holds
\begin{equation}\label{8}
i_{P}(B_{1})=i_{P}(B_{1})+\nu_{P}(B_{1})<i_{P}(B_{0}),
\end{equation}
and under the condition (\ref{7}), there holds
\begin{equation}\label{12}
i_{P}(B_{0})+\nu_{P}(B_{0})<i_{P}(B_{1}).
\end{equation}
We only prove (\ref{8}), the proof of (\ref{12}) is similar and we omit it here. By the condition (\ref{6}), we have
\begin{equation*}
i_{P}(B_{1})+\nu_{P}(B_{1})\leq i_{P}(B_{1}+lI_{2n})\leq i_{P}(B_{0}).
\end{equation*}
We shall prove that
\begin{equation}\label{9}
i_{P}(B_{1})<i_{P}(B_{1}+lI_{2n}).
\end{equation}

In fact, suppose $\gamma_{1}(t)\in P(2n)$ is a symplectic path which is the fundamental solution of the linear Hamiltonian system associated with the matrix $B_{1}(t)$. Since $JB_{1}(t)=B_{1}(t)J$, it is easy to verify that $\gamma: =\exp(Jlt)\gamma_{1}(t)$ is the fundamental solution of the linear Hamiltonian systems
\begin{equation*}
\dot{z}=J(B_{1}(t)+lI_{2n})z.
\end{equation*}
Note that we can regard  $W_{P}$ as
$$W_{P} = \{x \in W^{1/2,2}([0, 1], \R^{2n}) \mid x(t) = \gamma_{P}(t)\xi(t), \xi(t)\in W^{1/2,2}([0, 1], \R^{2n})\},$$
since $P$ is symplectic orthogonal, $P$ has the form $P = \exp{(M_{1})}$, the matrix $M_{1}$ satisfies $M^{T}_{1}J+JM_{1} = 0$ and $M^{T}_{1}+ M_{1}=0$. $\gamma_{P}(t)=\exp{(tM_{1})}$ as is defined in \cite{liutss1,liutss3}.
It has been proved in \cite{liutss1} that
\begin{equation}\label{10}
i_{P}(\gamma) - i_{P}(\gamma_{P}) =
i(\gamma_{P}(t)^{-1}\gamma(t)) + n,
\end{equation}
where $\gamma = \gamma(t)$ is the fundamental solution of $\dot{z}(t) = JB(t)z(t)$ with $P^{T}B(t+1)P=B(t)$, $i(\cdot)$ is the Maslov type index(cf.\cite{LO}). Set $\widetilde{B}_{\gamma_{P}}(t) = \gamma_{P}(t)^{T}J\dot{\gamma}_{P}(t) + \gamma_{P}(t)^{T}B_{1}(t)\gamma_{P}(t)$, we know that $\gamma_{P}(t)^{-1}\gamma_{1}(t)$ and $\gamma_{P}(t)^{-1}\gamma(t)$ are the fundamental solution of $\dot{z}(t) = J\widetilde{B}_{\gamma_{P}}(t)z(t)$ and $\dot{z}(t) = J(\widetilde{B}_{\gamma_{P}}(t)+lI_{2n})z(t)$ respectively.
For $l=2\pi$, the paths $\gamma_{P}(t)^{-1}\gamma(t)$ and $\gamma_{P}(t)^{-1}\gamma_{1}(t)$ have the same endpoint. Moreover, the rotation numbers satisfy
\begin{equation*}
\bigtriangleup_{1}(\gamma_{P}(t)^{-1}\gamma(t))=2n+\bigtriangleup_{1}(\gamma_{P}(t)^{-1}\gamma_{1}(t)).
\end{equation*}
Then we have
\begin{equation}\label{11}
i(\gamma_{P}(t)^{-1}\gamma_{1}(t))+2n\leq i(\gamma_{P}(t)^{-1}\gamma(t)).
\end{equation}
Finally by (\ref{10}) and (\ref{11}), we get
\begin{equation}\label{15}
i_{P}(B_{1})+2n<i_{P}(B_{1}+lI_{2n}).
\end{equation}
By the condition $(H_{\infty})$, $H^{\prime\prime}(t, x)$ is bounded and there exist $\mu_{1}$, $\mu>0$ such that
\begin{equation}\label{13}
I_{2n}\leq H^{\prime\prime}(t, x)+\mu I\leq\mu_{1}I_{2n},\ \ \forall\ (t, x).
\end{equation}
We define a convex function $N(t, x)=H(t, x)+\frac{1}{2}\mu\vert x\vert^{2}$. Its Fenchel dual $N^{\ast}(t, x)$ which is defined by
$$N^{\ast}(t, x)=\sup_{y\in\R^{2n}}\{(x, y)-N(t, x)\}$$
satisfying (cf.\cite{ek})
$$N^{\ast}(t, x)\in C^{2}(\R\times\R^{2n})$$
\begin{equation}\label{14}
N^{\ast\prime\prime}(t, y)=N^{\prime\prime}(t, x)^{-1},\ \ \text{for}\ \ y=N^{\prime}(t,x).
\end{equation}
From (\ref{13}) we have
\begin{equation}\label{25}
\mu_{1}^{-1}I_{2n}\leq N^{\ast\prime\prime}(t, y)\leq I_{2n}, \ \ \forall\ (t, y).
\end{equation}
So we have $\vert x\vert\to\infty$ if and only if $\vert y\vert\to\infty$ with $y=N^{\prime}(t, x)$. From the condition $(H_{\infty})$ and (\ref{14}), there exists $r_{1}$ such that
\begin{equation}\label{26}
(B_{2}(t)+\mu I_{2n})^{-1}\leq N^{\ast\prime\prime}(t, y)\leq (B_{1}(t)+\mu I_{2n})^{-1},\ \ \forall\ (t, y)\ \text{with}\ \vert y\vert\geq r_{1}.
\end{equation}
We choose $\mu>0$ satisfying (\ref{13}) and $\mu\notin\sigma(A)$. We recall that $(\Lambda_{\mu}x)(t)=-J\dot{x}(t)+\mu x(t)$. Consider the functional defined by
\begin{equation}\label{30}
f(u)=-\int_{0}^{1}[\frac{1}{2}(\Lambda_{\mu}^{-1}u(t), u(t))-N^{\ast}(t, u(t))]dt,\ \ \forall u\in L_{P},
\end{equation}
it is easy to see that $f\in C^{2}(L_{P}, \R)$. Next we prove that $f$ satisfies the Palais-Smale condition(cf.\cite{dong1,ek}).

Assume that $\{u_{j}\}$ is a sequence in $W_{P}$ such that $f(u_{j})$ is bounded and $f^{\prime}(u_{j})\to 0$.
By $(H_{0})$, we have $N^{\prime}(t, 0)=0$ and $N^{\ast\prime}(t, 0)=0$ and
\begin{equation}\label{24}
(f^{\prime}(u), v)=-\int_{0}^{1}[(\Lambda_{\mu}^{-1}u(t), v(t))-(N^{\ast\prime}(t, u(t)), v(t))]dt,\ \ \forall u, v\in L_{P}.
\end{equation}
Note that $\int_{0}^{1}N^{\ast\prime\prime}(t, \tau u(t))d\tau u(t)=N^{\ast\prime}(t, u(t))$, we have
\begin{equation}\label{21}
\Lambda_{\mu}^{-1}u_{j}(t)-\int_{0}^{1}N^{\ast\prime\prime}(t, \tau u_{j}(t))d\tau u_{j}(t)\to 0, \ \ \text{in}\ \ L_{P}.
\end{equation}
If $\Vert u_{j}\Vert_{2}\to\infty$, we set $x_{j}=u_{j}/\Vert u_{j}\Vert_{2}$. $L_{P}$ is a reflexive Hilbert space and $\Vert x_{j}\Vert_{2}=1$, $\forall j\in\N$, without loss of generality, we assume $x_{j}\rightharpoonup x_{0}$, and hence $\Lambda_{\mu}^{-1}x_{j}\to \Lambda_{\mu}^{-1}x_{0}$ in $L_{P}$. For any $\delta\in (0, 1)$ fixed, set
\begin{equation*}
C_{j}(t) =
 \begin{cases}
\int_{0}^{1}N^{\ast\prime\prime}(t, \tau u_{j}(t))d\tau, & \text{if}\ \vert u_{j}\vert\geq r_{1}/\delta,\\
(B_{1}(t)+\mu I_{2n})^{-1}, & \text{otherwise},
\end{cases}
\end{equation*}
\begin{equation*}
\eta_{j}(t)=\int_{0}^{1}N^{\ast\prime\prime}(t, \tau u_{j}(t))d\tau u_{j}(t)-C_{j}(t)u_{j}(t).
\end{equation*}

Then there exists a constant $M_{1}>0$ such that
\begin{equation}
\vert \eta_{j}\vert\leq M_{1}\ \ \text{for}\ a.e. \ t\in (0, k)
\end{equation}
and
\begin{equation*}
(1-\delta)(B_{2}(t)+\mu I_{2n})^{-1}+\delta I_{2n}\leq C_{j}(t)\leq (1-\delta)(B_{1}(t)+\mu I_{2n})^{-1}+\mu_{1}^{-1}\delta I_{2n}.
\end{equation*}
So for every $s>0$, there exists $\delta>0$ such that
\begin{equation}\label{22}
((B_{2}(t)+s I_{2n})+\mu I_{2n})^{-1}\leq C_{j}(t)\leq ((B_{1}(t)-sI_{2n})+\mu I_{2n})^{-1},\ \ \forall t\in (0, k).
\end{equation}
Now we may assume $C_{j}^{-1}u(t)\rightharpoonup B_{0}(t)u(t)$ in $L_{P}$ for $u\in L_{P}$ with $\mu I+B_{1}-\varepsilon I\leq B_{0}\leq \mu I_{2n}+B_{2}+sI_{2n}$. Let $\Lambda_{\mu}^{-1}x_{0}(t)=y_{0}(t)$, from (\ref{21})-(\ref{22}), we have
\begin{equation}\label{23}
J\dot{y_{0}}(t)+(B_{0}(t)-\mu I_{2n})y_{0}(t)=0,\ \ y_{0}(1)=Py_{0}(0).
\end{equation}
From the condition $(H_{\infty})$ and Theorem \ref{the:5} (2), for $s>0$ small enough, we have $\nu_{P}(B_{1}-sI_{2n})=\nu_{P}(B_{2}+sI_{2n})=0$ and $i_{P}(B_{1}-sI_{2n})=i_{P}(B_{2}+sI_{2n})$. So $\nu_{P}(B_{0}-\mu I_{2n})$ vanishes. This is impossible since $\Vert y_{0}\Vert_{2}=1$ and $y_{0}$ is a nontrivial solution of (\ref{23}). Hence $\Vert u_{j}\Vert_{2}$ is bounded.

Assume $u_{j}\rightharpoonup u_{0}$ in $L_{P}$, then $\Lambda_{\mu}^{-1}u_{j}\to \Lambda_{\mu}^{-1}u_{0}$. Let $\zeta_{j}: =\Lambda_{\mu}^{-1}u_{j}-N^{\ast\prime}(t, u_{j}(t))$, then $N^{\ast\prime}(t, u_{j}(t))=\Lambda_{\mu}^{-1}u_{j}-\zeta_{j}\to\Lambda_{\mu}^{-1}u_{0}$ by (\ref{21}). The Fenchel conjugate formula gives $u_{j}=N^{\prime}(\Lambda_{\mu}^{-1}u_{j}-\zeta_{j})\to N^{\prime}(\Lambda_{\mu}^{-1}u_{0})$. So $f$ satisfies the (P.S) condition.

There is a one-to-one correspondence from the critical points of $f$ to the solutions of the systems (\ref{1}).
Note that $0$ is a trivial critical point of $f$ and $N^{\ast\prime}(t, 0)=0$.
At every critical point $u_{0}$, the second variation of $f$ defines a quadratic form on $L_{P}$ by
\begin{equation}
(f^{\prime\prime}(u_{0})u, u)=-\int_{0}^{k}[(\Lambda_{\mu}^{-1}u(t), u(t))-(N^{\ast\prime\prime}(t, u_{0}(t))u(t), u(t))],\ \ \forall u\in L
_{P}.
\end{equation}
The critical point $u_{0}$ corresponds to a solution $x_{0}=\Lambda_{\mu}^{-1}u_{0}(t)$. By (\ref{14}), we have
\begin{equation}
N^{\ast\prime\prime}(t, u_{0}(t))=N^{\prime\prime}(t, x_{0}(t))^{-1}=(H^{\prime\prime}(t, x_{0})+\mu I_{2n})^{-1}.
\end{equation}
By definition, we have
\begin{equation}
m^{-}(f^{\prime\prime}(u_{0}))=i_{\mu}^{\ast}(B),\ \  m^{0}(f^{\prime\prime}(u_{0}))=\nu_{\mu}^{\ast}(B),\ \ \text{where}\ B(t)=H^{\prime\prime}(t, x_{0}).
\end{equation}
By Theorem \ref{the:3}, we have
\begin{equation*}
\nu_{\mu}^{\ast}(B)=\nu_{P}(B),\ \ i_{\mu}^{\ast}(B)=2mn+i_{P}(B)-M,
\end{equation*}
The index pair $(i_{P}(B), \nu_{P}(B))$ is the Maslov $P$-index of the linear Hamiltonian system
\begin{equation*}
\dot{y}(t)=JB(t)y(t).
\end{equation*}
By condition (\ref{6}) and the result (\ref{15}), we have
\begin{equation}\label{16}
i_{P}(B_{1})+\nu_{P}(B_{1})+2n<i_{P}(B_{0}).
\end{equation}
By condition (\ref{7}), similarly we have
\begin{equation}\label{17}
i_{P}(B_{0})+\nu_{P}(B_{0})+2n<i_{P}(B_{1}).
\end{equation}
From (\ref{16})-(\ref{17}), Theorem \ref{the:2} and Corollary \ref{coro:2}, we get that
\begin{equation}
\vert i_{P}(B_{0})-i_{P}(B_{1})\vert\geq 2n\ \ \text{and}\ \ \vert i_{\mu}^{\ast}(B_{0})-i_{\mu}^{\ast}(B_{1})\vert\geq 2n.
\end{equation}

Note that
\begin{equation*}
N^{\ast\prime\prime}(t, 0)=N^{\prime\prime}(t, 0)^{-1}=(H^{\prime\prime}(t, 0)+\mu I)^{-1}.
\end{equation*}
and $B_{0}(t)=H^{\prime\prime}(t, 0)$, so
\begin{equation*}
m^{-}(f^{\prime\prime}(0))=i_{\mu}^{\ast}(B_{0}),\ \  m^{0}(f^{\prime\prime}(0))=\nu_{\mu}^{\ast}(B_{0}).
\end{equation*}

Hence, by Lemma \ref{lem:1}, we only need to show the homology groups satisfy
\begin{equation}\label{18}
H_{q}(L_{P}, f_{a}; \R)\cong \delta_{q, r}\R, \ \ q=0,1,2,\cdots,
\end{equation}
for some $a\in\R$ and $r=i_{\mu}^{\ast}(B_{1})$. $f_{a}=\{x\in L_{P}\mid f(x)\leq a\}$ is the level set below $a$.  We proceed in three steps.

{\bf Step 1.} For $P\in Sp(2n)$ with $P^{T}P=I_{2n}$, $B_{j}(t)\in C(\R, \mathfrak{L}_{s}(\R^{2n}))$ satisfies $P^{T}B_{j}(t+1)P = B_{j}(t)$, $j=1,2$ and $B_{1}(t)<B_{2}(t)$, there holds
\begin{equation*}
L_{P}=L_{\mu}^{-}(B_{1})\oplus L_{\mu}^{+}(B_{2}),
\end{equation*}
where $L_{\mu}^{\ast}$ for $\ast=\pm,0$ is defined in Section 3.

In fact, if $0\neq u\in L_{\mu}^{-}(B_{1})$, then $Q_{l, B_{1}}^{\ast}(u)<0$,
\begin{equation*}
Q_{l, B_{2}}^{\ast}(u)\leq Q_{l, B_{1}}^{\ast}(u)<0,
\end{equation*}
and $u\notin L_{\mu}^{+}(B_{2})$. We only need to prove that $L_{P}=L_{\mu}^{-}(B_{1})+ L_{\mu}^{+}(B_{2})$.

By Theorem \ref{the:3}, $\nu_{\mu}^{\ast}=\nu_{P}(B_{2})=0$, we have $L_{P}=L_{\mu}^{-}(B_{2})\oplus L_{\mu}^{+}(B_{2})$. By $(H_{\infty})$ and Corollary \ref{coro:2}, we have $i_{\mu}^{\ast}(B_{1})=i_{\mu}^{\ast}(B_{2})=r$. Suppose $\xi_{1}, \xi_{2}, \dots, \xi_{r}$ be a basis of $L_{\mu}^{-}(B_{1})$. We have decompositions $\xi_{j}=\xi_{j}^{-}+\xi_{j}^{+}$ with $\xi_{j}^{-}\in L_{\mu}^{-}(B_{2})$ and $\xi_{j}^{+}\in L_{\mu}^{+}(B_{2})$. It is clear that $\{\xi_{j}^{-}\}_{j=1}^{r}$ is linear independent.
If $\sum_{j=1}^{r}\alpha_{j}\xi_{j}^{-}=0$, then $\bar{x}: =\sum_{j=1}^{r}\alpha_{j}\xi_{j}==\sum_{j=1}^{r}\alpha_{j}\xi_{j}^{+}\in L_{\mu}^{+}(B_{2})$, and $\bar{x}\in L_{\mu}^{-}(B_{1})$, so $\bar{x}=0$ and $\alpha=0, j=1,\dots, j$.

Since $\dim L_{\mu}^{-}(B_{2})=i_{\mu}^{\ast}(B_{2})=i_{\mu}^{\ast}(B_{1})=r$, $\{\xi_{j}^{-}\}_{j=1}^{r}$ is a basis of $L_{\mu}^{-}(B_{2})$. For any $\xi\in L_{P}$ written as $\xi=\xi^{-}+\xi^{+}$ with with $\xi^{-}\in L_{\mu}^{-}(B_{2})$ and $\xi^{+}\in L_{\mu}^{+}(B_{2})$, we have $\xi^{-}=\sum_{j=1}^{r}\beta_{j}\xi_{j}^{-}$. So
\begin{equation*}
\xi=\sum_{i=1}^{r}\beta_{j}\xi_{j}+(u^{+}-\sum_{i=1}^{r}\beta_{j}\xi_{j}^{+}),
\end{equation*}
the first sum lies in $L_{\mu}^{-}(B_{1})$ and the remainder is in $L_{\mu}^{+}(B_{2})$.

\smallskip
{\bf Step 2.} For sufficiently small $s>0$, we set $D_{R}: =L_{\mu}^{-}(B_{1}-sI_{2n})\oplus\{L_{\mu}^{+}(B_{2}+sI_{2n})\mid \Vert u\Vert\leq R\}$.
For $R>0$ and $-a>0$ large enough, we have the following deformation result:
\begin{equation}
H_{q}(L_{P}, f_{a}; \R)= H_{q}(D_{R}, D_{R}\cap f_{a}; \R),\ \ \text{for}\ \ q=0, 1, 2, \dots.
\end{equation}

In fact, from the condition $(H_{\infty})$ and Theorem \ref{the:5}, we have $\nu_{P}(B_{1}-sI_{2n})=\nu_{P}(B_{1})=0$, $\nu_{P}(B_{2}+sI_{2n})=\nu_{P}(B_{2})=0$ and then $i_{P}(B_{1}-sI_{2n})=i_{P}(B_{1})=i_{P}(B_{2})=i_{P}(B_{2}+sI)$.

By the condition $(H_{\infty})$ and Step 1, any $u\in L_{P}$ can be written as $u=u_{1}+u_{2}$ with $u_{1}\in L_{\mu}^{-}(B_{1}-sI_{2n})$ and $u_{2}\in L_{\mu}^{+}(B_{2}+sI_{2n})$, from (\ref{24}), we have
\begin{equation}\label{27}
\begin{split}
(f^{\prime}(u), u_{2}-u_{1})&=-\int_{0}^{k}[(\Lambda_{\mu}^{-1}u, u_{2}-u_{1})-(N^{\ast\prime}(t, u(t)), u_{2}-u_{1})]dt\\
&=\int_{0}^{k}(\Lambda_{\mu}^{-1}u_{1}, u_{1})dt-\int_{0}^{k}(\int_{0}^{1}N^{\ast\prime\prime}(t, \tau u(t))d\tau u_{1}, u_{1})dt-\int_{0}^{k}(\Lambda_{\mu}^{-1}u_{2}, u_{2})dt\\
&-\int_{0}^{k}(\int_{0}^{1}N^{\ast\prime\prime}(t, \tau u(t))d\tau u_{2}, u_{2})dt.
\end{split}
\end{equation}
By (\ref{25}) and (\ref{26}), we have
\begin{equation}
\begin{split}
\int_{0}^{k}(\int_{0}^{1}N^{\ast\prime\prime}(t, \tau u(t))d\tau u_{1}, u_{1})dt&=\int_{0}^{k}(\int_{0}^{h(t, u)}N^{\ast\prime\prime}(t, \tau u(t))d\tau u_{1}, u_{1})dt+\int_{0}^{k}(\int_{h(t, u)}^{1}N^{\ast\prime\prime}(t, \tau u(t))d\tau u_{1}, u_{1})dt\\
&\leq c_{0}\Vert u\Vert_{2}+\int_{0}^{k}(B_{1}(t)+\mu I_{2n}-sI_{2n})^{-1}u_{1}, u_{1})dt,
\end{split}
\end{equation}
where $h(t, u)=r_{1}/\vert u(t)\vert$. Similarly, we have
\begin{equation}\label{28}
\begin{split}
\int_{0}^{k}(\int_{0}^{1}N^{\ast\prime\prime}(t, \tau u(t))d\tau u_{2}, u_{2})dt&\geq\int_{0}^{k}\int_{h(t, u)}^{1}N^{\ast\prime\prime}(t, \tau u(t))d\tau u_{2}, u_{2})dt\\
&\geq\int_{0}^{k}(B_{2}(t)+\mu I_{2n}+sI_{2n})^{-1}u_{2}, u_{2})dt-c\Vert u\Vert_{2},\ \ \text{for}\ c>0.
\end{split}
\end{equation}
Note that in the subspace $L_{\mu}^{-}(B_{1}-sI_{2n})$ of $L_{P}$, the norm $\Vert\cdot\Vert_{2}$ is equivalent to $\Vert\cdot\Vert_{1}$ defined by
\begin{equation*}
\Vert\cdot\Vert_{1}: =(\int_{0}^{k}(B_{1}(t)+\mu I_{2n}-s I)^{-1}u_{1}, u_{1})dt)^{1/2}.
\end{equation*}
In this way, by (\ref{27})-(\ref{28})we obtain
\begin{equation}
(f^{\prime}(u), u_{2}-u_{1})\geq c_{1}\Vert u_{1}\Vert_{2}^{2}+c_{2}\Vert u_{2}\Vert_{2}^{2}-c_{3}(\Vert u_{1}\Vert_{2}+\Vert u_{2}\Vert_{2}).
\end{equation}
Thus, for large $R$ with $\Vert u_{1}\Vert_{2}\geq R$ or $\Vert u_{2}\Vert_{2}\geq R$, we have
\begin{equation}\label{29}
-(f^{\prime}(u), u_{2}-u_{1})<-1.
\end{equation}
We know from (\ref{29}) that $f$ has no critical point outside $D_{R}$ and $-f^{\prime}(u)$ points inwards to $D_{R}$ on $\partial D_{R}$. Therefore, we can define the define the deformation by negative flow. For any $u=u_{1}+u_{2}\notin D_{R}$, let $\sigma(t, u)=e^{\theta}u_{1}+e^{-\theta}u_{2}$, and $d_{u}=\ln \Vert u_{2}\Vert_{2}-\ln R$. We define the deformation map $\eta: [0, 1]\times L_{P}\to L_{P}$ by
\begin{equation*}
\eta(t, u_{1}+u_{2}) =
 \begin{cases}
u_{1}+u_{2}, & \text{if}\ \Vert u_{2}\Vert_{2}\leq R,\\
\sigma(d_{u}\theta, u), & \text{if}\ \Vert u_{2}\Vert_{2}> R.
\end{cases}
\end{equation*}
Then $\eta$ is continuous and satisfies
\begin{equation*}
\eta(0, \cdot)=id,\ \eta(1, L_{P})\subset D_{R},\ \eta(1, f_{a})\subset D_{R}\cap f_{a},
\end{equation*}
\begin{equation*}
\eta(\theta, f_{a})\subset f_{a}, \ \ \eta(\theta, \cdot)\mid_{D_{R}}=id\mid_{D_{R}}.
\end{equation*}
Hence the pair $(D_{R}, D_{R}\cap f_{a})$ is a deformation retract of the pair $(L_{P}, f_{a})$.

\smallskip
{\bf Step 3.} For $R, -a>0$ large enough, there holds
\begin{equation}
H_{q}(D_{R}, D_{R}\cap f_{a}; \R)\cong \delta_{q, r}\R, \ \ q=0,1,2,\cdots,
\end{equation}

In fact, similarly to the above computation, for a large number $m>0$, we have
\begin{equation*}
\begin{split}
&\int_{0}^{k}N^{\ast}(t, u(t))dt\\
&=\int_{0}^{k}(\int_{0}^{1}\tau d\tau\int_{0}^{1}(N^{\ast\prime\prime}(t, \tau su(t))dsu(t), u(t)))dt
+\int_{0}^{k}N^{\ast}(t, 0)dt\\
&\leq\int_{\vert u(t)\vert\geq m r_{1}}(\int_{0}^{1}\tau d\tau\int_{0}^{1}(N^{\ast\prime\prime}(t, \tau su(t))dsu(t), u(t))dt+c_{m}\\
&\leq\int_{\vert u(t)\vert\geq m r_{1}}(\int\int_{\vert\tau su(t)\vert\geq r_{1}, \tau, s\in [0, 1]} \tau N^{\ast\prime\prime}(t, \tau su(t))dsd\tau u(t), u(t))dt\\
&+\int_{\vert u(t)\vert\geq m r_{1}}(\int\int_{\vert\tau su(t)\vert\leq r_{1}, \tau, s\in [0, 1]} \tau N^{\ast\prime\prime}(t, \tau su(t))dsd\tau u(t), u(t))dt+c_{m}\\
&\leq \frac{1}{2}\int_{0}^{k}(B_{1}(t)+\mu I_{2n})^{-1}u(t), u(t))dt+d_{m}\Vert u_{2}\Vert_{2}+c_{m},
\end{split}
\end{equation*}
where $c_{m}$ and $d_{m}$ are constants depending only on $m$ and $d_{m}\to 0$ as $m\t\infty$. Hence for the small $s$ in Step 2 above, we can choose a large number $m$ such that
\begin{equation*}
\int_{0}^{k}N^{\ast}(t, u(t))dt\leq\frac{1}{2}\int_{0}^{k}(B_{1}(t)+\mu I_{2n}-sI)^{-1}u(t), u(t))dt+C\ \ \forall u\in L_{P}
\end{equation*}
for some constant $C>0$. Together with (\ref{30}), this yields, for any $u=u_{1}+u_{2}$ with $u_{1}\in L_{\mu}^{-}(B_{1}-sI_{2n})$ and $u_{2}\in L_{\mu}^{+}(B_{2}+sI_{2n})$ with $\Vert u_{2}\Vert_{2}\leq R$, we have
\begin{equation*}
f(u)\leq -C_{1}\Vert u_{1}\Vert_{2}^{2}+C_{2}\Vert u_{1}\Vert_{2}+C_{3},
\end{equation*}
where $C_{j}$, $j=1,2,3$ are constants and $C_{1}>0$. It implies that $f(u)\to -\infty$ if and only if $\Vert u_{1}\Vert_{2}\to\infty$ uniformly for $u_{2}\in L_{\mu}^{+}(B_{2}+sI_{2n})$ with $\Vert u_{2}\Vert_{2}\leq R$. In the following we denote $B_{r}=\{x\in L_{P}\mid \Vert x\Vert_{2}\leq r\}$ the ball with radius $r$ in $L_{P}$. Thus there exist $T>0$, $a_{1}<a_{2}<-T$, and $R<R_{1}<R_{2}<R_{3}$ such that
\begin{equation*}
\begin{split}
&(L_{\mu}^{+}(B_{2}+sI_{2n})\cap B_{R_{3}}\oplus (L_{\mu}^{-}(B_{1}-sI_{2n})\backslash B_{R_{2}}\subset f_{a_{1}}\cap D_{R_{3}}\\
&\subset (L_{\mu}^{+}(B_{2}+sI_{2n})\cap B_{R_{3}}\oplus (L_{\mu}^{-}(B_{1}-sI_{2n})\backslash B_{R_{1}}\subset f_{a_{2}}\cap D_{R_{3}}.
\end{split}
\end{equation*}
For any $u\in D_{R_{3}}\cap (f_{a_{2}}\backslash f_{a_{1}})$, since $\sigma(t, u)=e^{\theta}u_{1}+e^{-\theta}u_{2}$, the function $f(\sigma(t, u))$ is continuous in $t$ and satisfies $f(\sigma(\theta, u))=f(u)>a_{1}$ and $f(\sigma(t, u))\to -\infty$ as $t\to +\infty$. It implies that there exists $\theta_{0}=\theta_{0}(u)>0$ such that $f(\sigma(\theta_{0}, u))=a_{1}$. But by (\ref{29}),
\begin{equation*}
\frac{d}{d\theta}f(\sigma(t, u))\leq -1,\ \ \text{at any point}\ \theta>0.
\end{equation*}
By the implicit function theorem, $\theta_{0}(u)$ is continuous in $u$. We define another deformation map $\eta_{0}: [0, 1]\times f_{a_{2}}\cap D_{R_{3}}\to f_{a_{2}}\cap D_{R_{3}}$ by
\begin{equation*}
\eta_{0}(\theta, u) =
 \begin{cases}
u, & \text{if}\ f_{a_{1}}\cap D_{R_{3}},\\
\sigma(\theta_{0}(u), u), & \text{if}\ u\in D_{R_{3}}\cap (f_{a_{2}}\backslash f_{a_{1}}).
\end{cases}
\end{equation*}
It is clear that $\eta_{0}$ is a deformation from $f_{a_{2}}\cap D_{R_{3}}$ to $f_{a_{1}}\cap D_{R_{3}}$. We now define
\begin{equation*}
\widetilde{\eta}(u)=d(\eta_{0}(1, u))\ \ \text {with}\ \ d(u)=
\begin{cases}
u, & \Vert u_{1}\Vert_{2}\geq R_{1},\\
u_{2}+\frac{u_{1}}{\Vert u_{1}\Vert_{2}}R_{1}, & 0<\Vert u_{1}\Vert_{2}< R_{1}.
\end{cases}
\end{equation*}
This map defines a strong deformation retract:
\begin{equation*}
\widetilde{\eta}: D_{R_{3}}\cap f_{a_{2}}\to L_{\mu}^{+}(B_{2}+sI_{2n})\cap B_{R_{3}})\oplus (L_{\mu}^{-}(B_{1}-sI_{2n})\cap\{u\in L_{P}\mid \Vert u\Vert_{2}\geq R_{1}\}).
\end{equation*}
Now we can compute the homology groups
\begin{equation*}
\begin{split}
H_{q}(D_{R_{3}}, D_{R_{3}}\cap f_{a_{2}}; \R)&\cong H_{q}(D_{R_{3}}, L_{\mu}^{+}(B_{2}+sI_{2n})\cap B_{R_{3}})\oplus (L_{\mu}^{-}(B_{1}-sI_{2n})\cap\{u\in L_{P}\mid \Vert u\Vert_{2}\geq R_{1}\}); \R)\\
&\cong H_{q}(L_{\mu}^{-}(B_{1}-sI_{2n})\cap B_{R_{3}}, \partial (L_{\mu}^{-}(B_{1}-sI_{2n})\cap B_{R_{3}}); \R)\\
&\cong  \delta_{q, r}\R.
\end{split}
\end{equation*}

\end{proof}

\begin{remark}
The method of the proof  (\ref{18}) comes from \cite{chang1}, but we have modified it to suit our case.
\end{remark}

\begin{corollary}
Suppose that $P \in Sp(2n)$ satisfies $P^{T}P=I_{2n}$, $H$ satisfies conditions (H), $(H_{0})$, $(H_{\infty})$. Suppose $B_{0}(t)=H^{\prime\prime}(t, 0)$ satisfying one of the following twisted conditions:
\begin{enumerate}

\item[(I)]  $B_{1}(t)<B_{0}(t)$, there exists $\lambda\in (0, 1)$ such that $\nu_{P}((1-\lambda)B_{1}+\lambda B_{0})\neq 0$;

\smallskip
\item[(II)]  $B_{0}(t)<B_{1}(t)$, there exists $\lambda\in (0, 1)$ such that $\nu_{P}((1-\lambda)B_{0}+\lambda B_{1})\neq 0$.

\end{enumerate}
Then the system (\ref{1}) possesses at least one non-trivial {\it $P$-solution}. Furthermore, if $\nu_{P}(B_{0})=0$ and in, we replace the second condition by
\begin{equation*}
\sum_{\lambda\in (0, 1)}\nu_{P}((1-\lambda)B_{1}+\lambda B_{0})\geq 2n,
\end{equation*}
or in , we replace the second condition by
\begin{equation*}
\sum_{\lambda\in (0, 1)}\nu_{P}((1-\lambda)B_{0}+\lambda B_{1})\geq 2n.
\end{equation*}
Then the system (\ref{1}) possesses at least two non-trivial {\it $P$-solutions}.
\end{corollary}
\begin{proof}
The proof follows from Lemma \ref{lem:2}, Lemma \ref{lem:1} and the above proof of Theorem \ref{the:4}. In the first case, we have $r=i_{P}(B_{1})\notin [i_{P}(B_{0}), i_{P}(B_{0})+\nu_{P}(B_{0})]$. In the second case we have $\vert i_{P}(B_{0})-i_{P}(B_{1})\vert\geq 2n$.

\end{proof}

The proof of Theorem \ref{the:4} in fact proves the following result.
\begin{theorem}
Suppose that $P \in Sp(2n)$ satisfies $P^{T}P=I_{2n}$, $H$ satisfies conditions (H), $(H_{0})$, $(H_{\infty})$. Suppose $B_{0}(t)=H^{\prime\prime}(t, 0)$ satisfying the following twisted condition:
\begin{equation}
i_{P}(B_{1})\notin [i_{P}(B_{0}), i_{P}(B_{0})+\nu_{P}(B_{0})].
\end{equation}
Then the system (\ref{1}) possesses at least one non-trivial {\it $P$-solution}. Furthermore, if $\nu_{P}(B_{0})=0$ and $\vert i_{P}(B_{1})-i_{P}(B_{0})\vert\geq 2n$, the system (\ref{1}) possesses at least two non-trivial {\it $P$-solutions}.

\end{theorem}

\end{document}